\documentclass[11pt]{amsart}
\usepackage{amssymb,amsmath}
\usepackage{subfigure}
\usepackage{graphpap,latexsym,epsf}
\usepackage{color,graphicx,psfrag}
\usepackage[colorlinks=true]{hyperref}
\hypersetup{urlcolor=blue, citecolor=red}

\newtheorem{proposition}{Proposition}[section]
\newtheorem{theorem}[proposition]{Theorem}

\newtheorem{lemma}[proposition]{Lemma}
\theoremstyle{definition}

\newtheorem{remark}[proposition]{Remark}

\numberwithin{equation}{section}

\def\E{\boldsymbol{\eps}}
\def\Ik{I^{\lambda,\vk}_{\E}}
\def\k{\varkappa}
\def\vk{\varkappa}

\def\eps{\varepsilon}

\def\R {\mathbb{R}}
\def\N {\mathbb{N}}

\def\U {{\mathcal U}}

\def \and{\quad\text{and}\quad}

\def \no#1#2#3 {{\bf #1} (#3), #2.}
\def \eds#1#2#3 {#1, #2, #3.}
\title[Global minimizers of coexistence for competing species]{Global minimizers of coexistence for competing species}

\author[Monica Conti and Veronica Felli]{}

 \email{monica.conti@polimi.it}
 \email{veronica.felli@unimib.it}

\begin{document}
\maketitle

\centerline{\scshape Monica Conti} \medskip {\footnotesize
  \centerline{Politecnico di Milano - Dipartimento di Matematica ``F.\
    Brioschi''} \centerline{Via Bonardi 9, 20133 Milano, Italy} }
\bigskip
\centerline{\scshape Veronica Felli} \medskip {\footnotesize
  \centerline{Universit\`a di Milano Bicocca - Dipartimento di
    Matematica e Applicazioni} \centerline{Via Cozzi 53, 20125
    Milano, Italy} }

\bigskip

\begin{abstract}
  A class of variational models describing ecological systems of $k$
  species competing for the same resources is investigated.  The
  occurrence of coexistence in minimal energy solutions is discussed and positive
  results are proven for suitably differentiated internal dynamics.
\end{abstract}

\section{Introduction}
This paper is focused on a class of variational problems suitable for
studying the dynamic of segregation of $k$ organisms which share the
same territory $\Omega\subset\R^N$. Calling $u_i$ the density of the
$i$-th population and $F_i(u_i)$ its internal potential, the \emph{free
  energy} of the system is
\begin{equation}\label{eq:freeenergy}
  {\mathcal E}(u_1,u_2,\dots,u_k)=\sum_{i=1}^{k} \int_\Omega\bigg(\frac12|\nabla u_i(x)|^2-F_i(u_i(x))\bigg)dx,
\end{equation}
given by the sum of the internal energies of each species.
In this context, the question of finding
  a global minimizer of the energy in the class of segregated
  states arises in a natural way. More precisely, if we define
\[\mathcal U=\left\{U=(u_1,u_2,\dots,u_k) \in
  [H^1(\Omega)\big]^k:\, u_i\geq 0,\ u_i\cdot u_j=0\text{ if
  }i\neq j,\text{ a.e. in }\Omega\right\},\]
  we are led to the following \emph{optimal partition problem}:
\begin{equation}\label{eq:opp}
\text{finding $U\in \U$ such that }\;{\mathcal E}(U)=\min_{V\in\U} {\mathcal E}(V).
\end{equation}
This problem has been recently settled in \cite{ctv-var},
in connection with strongly competing variational systems
of Lotka-Volterra type
\begin{equation}\label{eq:sistema}
-\Delta u_i= f_i(u_i)-\vk \;u_i\sum_{j\neq i}u_j^2,\quad\text{ in } \Omega,
\end{equation}
which, since the pioneering work of Volterra, constitute one of the
most studied theoretical models of population ecology, see \cite{MV}.
As a matter of fact, as the competition rate $\vk$ grows indefinitely,
the components of any (nonnegative) solution of the system tend to
separate their supports, leading to an element of $\mathcal U$; in
particular, the problem of finding minimal energy solutions of
\eqref{eq:sistema} formally translates, as $\vk\to\infty$, into
\eqref{eq:opp}, see  also \cite{ctv2,cdhmn,dhmp}.

In the understanding of the spatial behavior of interacting species, a
central problem is to establish whether \emph{coexistence} of all the
species occurs, or the internal growth leads to \emph{extinction}, that
is, configurations where one or more densities are null: in this paper
we address the question in the two different theoretical settings, endowing
the models with
null Dirichlet boundary conditions:
\begin{equation}\label{eq:dirichlet}
u_i=0\quad \text{on }\partial\Omega,\qquad i=1,\dots,k.
\end{equation}
At a first insight, extinction has to be expected for competing
systems which are, in a sense, too uniform.  For instance, in the case
of null Neumann boundary conditions, the global minimum of ${\mathcal
  E}$ on $\U$ is in general achieved by configurations where only one
species is alive, see \cite[Proposition 2.1]{cf-neumann}.
Nonetheless, a mechanism to avoid extinction can be found in the
spatial inhomogeneity of the territory. Indeed, working in a special
class of non-convex domains close to a union of $k$ disjoint balls,
the existence of \emph{local} minima of ${\mathcal E}$ where all the
species are present is proven in \cite{cf-neumann}, (see also
\cite{cf}).

As a matter of fact, if the internal energies $f_i$ are not
differentiated, extinction of global minimizers under null boundary
conditions occurs in any domain, see Section \ref{pf:nonexhistence}:
\begin{theorem}\label{th:nonexistence}
Let $\Omega$ be a  bounded Lipschitz domain and $f$ be a Lipschitz continuous function.
If  $f_i=f$ for all $i=1,\dots,k$ and $F_i(s)=\int_0^sf_i(t)\,dt$,  then any global
minimizer of ${\mathcal E}$
on ${\U}\cap H^1_0(\Omega)$ has at most one nonzero component.
 \end{theorem}
 This motivates the question whether different internal laws might
 produce a mechanism to ensure coexistence.  With the aim of providing
 a first answer to this conjecture, in this note we consider the
 special situation
 when the internal energies $f_i$ are of the same type but act at different density scales, see assumption \eqref{eq:lefi}.\\
 We first investigate global minimizers for systems in the form
 \eqref{eq:sistema}, namely solutions of the energy minimization
 problem
\begin{align*}
  \min\bigg\{&(u_1,u_2,\dots,u_k) \in
  [H^1(\Omega)\big]^k\;:\; \\
  &\sum_{i=1}^k\int_\Omega\bigg(\frac12|\nabla
  u_i(x)|^2- F_i(u_i(x))\bigg)dx+\frac12\vk\sum_{\substack{i,j=1\\i\neq j}}^k u_i(x)^2u_j(x)^2\bigg\}.
\end{align*}
We prove in Theorem \ref{th:main-sistema} that any global minimizer is a coexistence state of
\eqref{eq:sistema} where all the $k$ species are present, provided the internal growths $f_i$ of $k-1$ populations act at a small density
scale, depending on $\vk$.
This is done with a great deal of generality both with respect to the domain and to the competing interaction term
appearing in \eqref{eq:sistema}, see (H1)--(H3) below, but the dependence of $f_i$'s on $\vk$ does not allow recovering a meaningful coexistence result
for the corresponding optimal partition problem \eqref{eq:opp}, see Remark \ref{rem:mod_cont}.
It is worth pointing out that the investigation of positive solutions to
competitive systems in the case of $k\geq 3$ densities is a challenging
task and only partial results are known, see
e.g. \cite{cf,cf-neumann,ctv2,dd1,dd2,lmn} and the discussions therein
for more references.

To investigate the possibility of coexistence for solutions to the optimal partition problem (\ref{eq:opp}),
following an idea developed in \cite{sweers} to
show the existence of sign-changing solutions to some elliptic
equations, we focus on  a certain class of (possibly convex) domains characterized by the
presence of an angle.
In this framework we prove, in Theorem \ref{th:main-partition}, that any
 global minimizer of ${\mathcal E}$ among segregated states has \emph{at
 least two} nontrivial positive components.  Although the result is not
 exhaustive for an arbitrary number of species, for systems of two
 populations it allows us to provide the full picture of the coexistence
 phenomenon (Theorem \ref{th:system2}).  Namely, we first prove that any minimal solution of system \eqref{eq:sistema} with $k=2$ is an equilibrium configuration where both the species are present, provided the scales of their internal energies are different but \emph{independent} from the competition rate $\vk$. Hence we perform the
 asymptotic analysis as $\vk\to\infty$ and prove that both
 components  survive as the interspecific competition becomes larger and
 larger. As a result, any minimal state of system \eqref{eq:sistema}
 converges to a spatially segregated distribution where the two
 densities coexist and solve the optimal partition problem
 \eqref{eq:opp}.

 In conclusion, our results suggest that in ecological systems
 with strong competition between the species, suitably differentiated internal energies may ensure
 coexistence in minimal energy configurations.

\section{Assumptions and main results}\label{sec:results}
Let $k\geq 2$,  $\E=(\eps_2,\dots,\eps_k)\in(0,1)^{k-1}$, $\lambda>0$,
$\vk >0$, and $\Omega$ be an open bounded set in $\R^N$ ($N\geq 2$).
We shall consider a class of competitive systems of the form
\begin{equation}\label{eq:sistema_con_H}
\begin{cases}
  -\Delta u_i(x)=\lambda f_{i,\E}(u_i(x))-\vk
\displaystyle\frac{\partial H}{\partial u_i}
  (u_1(x),u_2(x),\dots,
  u_k(x)),&\text{ in } \Omega,\\[7pt]
  u_i\in H^1_0(\Omega),\qquad i=1,\dots,k,&
\end{cases}
\end{equation}
where  $f_{i,\E}$ and $H$ satisfy the following sets of assumptions.

\medskip\noindent {\bf Assumptions on $f_{i,\E}$.}
Let $g\in C^0(\R)$  satisfying
\begin{itemize}
\item[(F1)] $g(s)=0$ for all $s\in(-\infty,0]$ and $g$ is
  right differentiable at $0$ with $g'_+(0)=1$;
\item[(F2)] there exists $\beta>0$ such that
$$
g(t)<0 \text{ for all } t>\beta \text{ and }
g(t)\geq0 \text{ for all }t\in (0,\beta);
$$
\item[(F3)] $\int_0^\beta g(s)ds=\alpha>0$.
\end{itemize}
A typical example is given by the classical logistic
  nonlinearity (see e.g. \cite{dd1}), namely $g(s)=s-s^2$ for $s\geq 0$.
We set
\begin{equation}\label{eq:lefi}
f_{i,\E}(s)=
\begin{cases}
  g(s),&\text{if }i=1,\\
  \frac{1}{\sqrt{k}\eps_i}g\left(\frac{\sqrt{k}}{\eps_i}t\right),&\text{if
  }i=2,\dots,k.
\end{cases}
\end{equation}
It is immediate to check that, for all $i\geq 2$, $f_{i,\E}$ satisfies
\begin{equation}
  \label{e:ipotesiI}
  f_{i,\E}(s)< 0 \mbox{ for all } s>\beta_i,\qquad \int_0^{\beta_i} f_{i,\E}(s)ds=
  \frac{\alpha}{k},
\end{equation}
where, for $i\geq 2$, $\beta_i=  \frac{\beta\eps_i}{\sqrt{k}}$.
For the sake of convenience, we shall refer to $\beta$ as~$\beta_1$.\\
\begin{figure}[h]
 \centering
     \includegraphics[width=7.5cm]{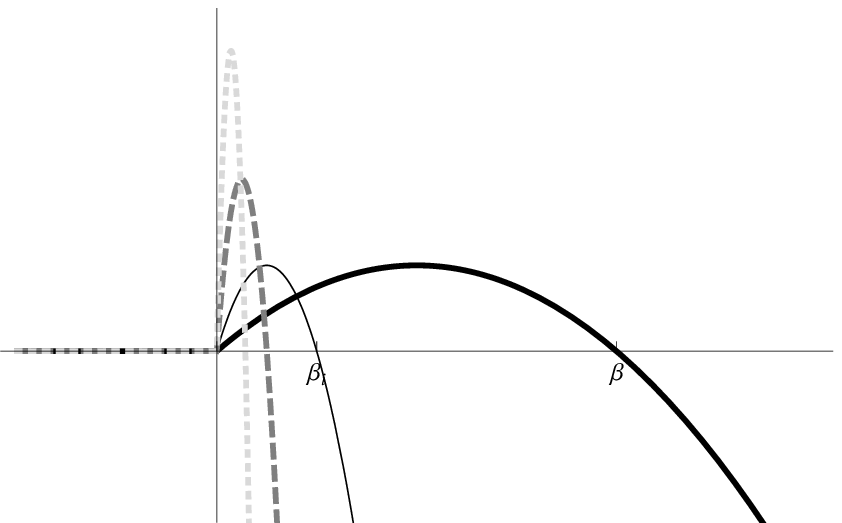}
 \caption{The nonlinearities $f_{i,\E}$ in the case $g(s)=s-s^2$
with $k=4$ and  $\E=(\frac12,\frac14,\frac17)$.
}\label{fig:nonlin}
\end{figure}
\medskip\noindent {\bf Assumptions on $H$.}
Let $H\in
C^1(\R^k)$ satisfy, for all $(s_1,s_2, \dots,s_k)\in \R^k$,
\begin{align}
  \tag{H1}
  &H(s_1,s_2, \dots,s_k)\geq 0,\\
  \tag{H2} &s_i\frac{\partial H}{\partial s_i}(s_1,s_2, \dots,s_k)\geq
  0
  \quad\text{for all }i=1,\dots,k,\\[5pt]
  \tag{H3}
&\left\{\hskip-10pt
\begin{array}{ll}
&H(s_1,s_2, \dots,s_k)=0\quad\text{if }s_i=0\text{ for at
    least $k-1$ variables},\\[5pt]
&\dfrac{\partial H}{\partial
    s_i}(s_1,s_2, \dots,s_k)=0\text{ implies that either }s_i=0\text{ or }
  s_j=0\text{ for all }j\neq i.
\end{array}
\right.
\end{align}
The first assumption states the competitive character of the
interaction term; a typical example which fits all the above
assumptions is
\begin{equation}\label{eq:HH}
H(s_1,s_2, \dots,s_k)=\frac12\sum_{\substack{i,j=1\\i\neq j}}^k s_i^2s_j^2,
\end{equation}
which is widely used in modeling population dynamics,
nonlinear optics (see e.g. \cite{AC,MMP}), and Bose-Einstein condensation
(see \cite{bose,Noris,TV}).

Setting $F_{i,\E}(t)=\int_0^t f_{i,\E}(s)ds$,
we define the internal energy of the system
 $$
I^{\lambda,\vk}_{\E}:[H^1_0(\Omega)]^k\to(-\infty,\infty]
$$
as
\begin{align}\label{eq:Ikappa}
  I^{\lambda,\vk}_{\E}(u_1,\dots,u_k)=&\sum_{i=1}^k\int_\Omega\bigg(\frac12|\nabla
  u_i(x)|^2-\lambda F_{i,\E}(u_i(x))\bigg)dx\\\nonumber&+\vk \int_{\Omega}
  H(u_1(x),u_2(x),\dots,u_k(x))\,dx.
\end{align}

Our first result states that the global minimizers of
$I^{\lambda,\vk}_{\E}$ are configurations of coexistence if the
range $\E$ of the internal growths of $k-1$ species is suitably
related to $\lambda$ and $\vk$.

\begin{theorem}\label{th:main-sistema}
  Let $g\in C^0(\R)$ satisfy (F1--3), $H\in C^1(\R^k)$ satisfy
  (H1)--(H2), and  $\Omega\subset \R^N$ be a bounded domain.
  There exists $\lambda_0>0$ such that, for every
  $\lambda>\lambda_0$ and $\varkappa\geq 0$, there exists
  $\eps_{\lambda,\vk}>0$ with the following property:
  for all $\E\in (0,\eps_{\lambda,\vk})^{k-1}$, the competing system
  (\ref{eq:sistema_con_H}) with $f_{i,\E}$ as in \eqref{eq:lefi} has a solution
  $U=(u_1,\dots,u_k)\in [H^1_0(\Omega)]^k$ satisfying
 \begin{itemize}
 \item[(i)]  $u_i\not\equiv 0$ for all  $i=1,\dots,k$;
 \item[(ii)] $0\leq u_i\leq \beta_i$ for all $i=1,\dots,k$;
 \item[(iii)] $U$ is a global minimizer of $I^{\lambda,\vk}_{\E}$, namely
$$
I^{\lambda,\vk}_{\E}(U)=\min\left\{ I^{\lambda,\vk}_{\E}(V),\;
  V\in [H^1_0(\Omega)]^k \right\}.
$$
\end{itemize}
Furthermore, $\eps_{\lambda,\vk}$ depends on the ratio $\lambda/\vk$
and tends to 0 if $\lambda/\vk\to 0$.
\end{theorem}
Some remarks are in order.
\begin{remark}\label{rem:mod_cont}$\;$\\
  a) It will be clear from the proof how $\eps_{\lambda,\vk}$ depends
  on the data of the problem, see \eqref{eq:epsi}. For instance, for
  the Lotka-Volterra model in a ball, $H$ as in \eqref{eq:HH} and
  $g(u)=u-u^2$, we can choose
$$\eps_{\lambda,\vk}^2=\frac{1}{6k^2}\frac{\lambda}{\vk}.$$
b) If the interspecific competition rate $\vk$ grows (at $\lambda$
fixed), then every $\eps_i$ becomes smaller and smaller. Hence by (ii)
we learn that $k-1$ components annihilate uniformly in $\Omega$,
implying that in the  limit configuration as $\vk\to\infty$ only the first component is alive.\\
\end{remark}

Concerning the optimal partition problem stated in the introduction,
we shall focus on a special class of domains.\\

\medskip\noindent {\bf Description of the domain.}  Let $\Omega\subset
\R^N$, $N\geq 2$, be a bounded Lipschitz domain with ${\bf
  0}\in\partial\Omega$ such that
\begin{itemize}
\item[(D1)] $\Omega\subset T$, where $T=\Big\{x=(x_1,x_2,\dots,x_N)\in\R^N\,:\, m\sqrt{\sum_{i=2}^N x_i^2}\leq
  x_1\leq 1\Big\}$ and $m>1$;
\item[(D2)] there exists
$\delta_0\in (0,1)$ such that
$$
\delta\Omega=\{x\in\R^N:\delta^{-1}x\in\Omega\}\subset\Omega,
$$
for every $\delta\in(0,\delta_0)$.
\end{itemize}
\begin{figure}[h]
 \centering
   \begin{psfrags}
     \psfrag{O}{$\Omega$}
     \psfrag{T}{$T$}
     \includegraphics[width=6cm]{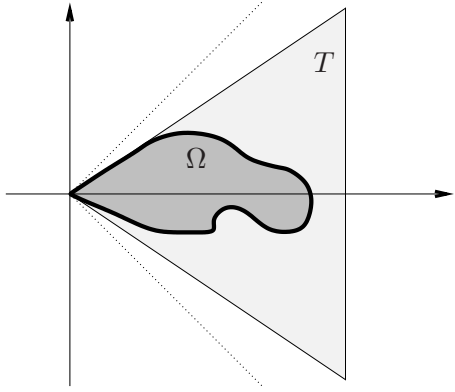}
   \end{psfrags}
 \caption{An example of domain $\Omega\subset\R^2$.}\label{fig:dd}
\end{figure}

\begin{theorem}\label{th:main-partition}
  Let $g\in C^0(\R)$ satisfy (F1--3) and $\Omega$ be a bounded
  Lipschitz domain satisfying (D1--2).  There exist $\lambda_0>0$ and
  $\eps_0>0$ such that, if $\E\in(0,\eps_0)^{k-1}$ and
  $\lambda>\lambda_0$, then every global minimizer of
\begin{equation*}
  {\mathcal E}^{\lambda}_{\E}(u_1,\dots,u_k)=\sum_{i=1}^k\int_\Omega\bigg(\frac12|\nabla u_i(x)|^2-
  \lambda F_{i,\E}(u_i(x))\bigg)dx
\end{equation*}
on ${\U}\cap H^1_0(\Omega)$ has at least  two nonnegative and nonzero  components.
\end{theorem}
In the case of two species we have a better understanding of the
phenomenon. In particular, we establish the link between limit configurations
of system (\ref{eq:sistema_con_H}) as $\vk\to\infty$ and solutions to the
optimal partition problem.
\begin{theorem}\label{th:system2}
  Let $k=2$. Let $g\in C^0(\R)$ satisfy (F1--3), $H\in C^1(\R^2)$
  satisfy (H1--3) and $\Omega$ be a bounded Lipschitz domain
  satisfying (D1--2).  There exist $\lambda_0>0$ and $\eps_0>0$ such
  that, if $\E\in(0,\eps_0)^{k-1}$ and $\lambda>\lambda_0$, then the
  global minimum of $\Ik$ is achieved for all $\vk>0$ and every
  global minimizer is a nontrivial configuration $(u_1^\vk,u_2^\vk)$
  with both $u_i^\vk\geq 0$, $u_i^\vk\not\equiv 0$, $i=1,2$. Moreover,
  for every fixed $\E\in(0,\eps_0)^{k-1}$ and $\lambda>\lambda_0$,
there exists $U=(u_1,u_2)\in [H^1_0(\Omega)]^2$ such that
  \begin{enumerate}
  \item $u_i\not\equiv 0$ for $i=1,2$,
  \item $U=(u_1,u_2)\in  \mathcal{U}$,
  \item $U$ is a global minimizer of ${\mathcal E}_{\E}^{\lambda}$ on
    ${\U}\cap H^1_0(\Omega)$,
  \end{enumerate}
  and, up to subsequences, $u_i^{\varkappa}$ converges strongly to $u_i$ in
  $H^1(\Omega)$.
\end{theorem}

\section{Preliminary results}
Let $\Omega$ be a bounded open set in $\R^N$ and $g\in C^0(\R)$
satisfy (F1--3). For every $\lambda>0$, $\E\in(0,1)^{k-1}$, and
$i\geq2$, let us define
$J_{1}^{\lambda},J_{i,\E}^{\lambda}:H^1_0(\Omega)\to
(-\infty,+\infty]$
\begin{align*}
  &J_{1}^{\lambda}(u)=\int_\Omega\bigg(\frac12|\nabla
  u(x)|^2-\lambda\int_0^{u(x)}g(s)\,ds\bigg)dx,\\
&J_{i,\E}^{\lambda}(u)=\int_\Omega\bigg(\frac12|\nabla
  u(x)|^2-\lambda F_{i,\E}(u(x))\bigg)dx.
\end{align*}
If $\lambda >\lambda_1(\Omega)$, with $\lambda_1(\Omega)$ being the
first eigenvalue of the Laplace operator with null Dirichlet boundary
conditions, it is easy to prove that the infima
$$
m_{1}^{\lambda}:=\inf\left\{ J_{1}^{\lambda}(u),\;u\in
  H^1_0(\Omega)\right\}, \quad m_{i,\E}^{\lambda}:=\inf\left\{
  J_{i,\E}^{\lambda}(u),\;u\in H^1_0(\Omega)\right\}
$$
are achieved and any minimizer is a positive weak solution to the
elliptic equation
$$
\begin{cases}
-\Delta u= \lambda f_{i,\E}(u), &\text{in }\Omega,\\
u\in H^1_0(\Omega),
\end{cases}
$$
see e.g. \cite{am}.
\begin{lemma}\label{l:limiti}
 There hold
\begin{align*}
  & m_{1}^{\lambda}\geq -\alpha\lambda|\Omega|,\quad
\lim_{\lambda\to+\infty} \lambda^{-1}m_{1}^{\lambda}=-\alpha |\Omega|,\\
  & m_{i,\E}^{\lambda}\geq -\lambda\frac{\alpha}k|\Omega|,\quad
\lim_{\lambda\to+\infty}
  \lambda^{-1}m_{i,\E}^{\lambda}=-\frac{\alpha}{k}|\Omega|,\qquad
  i=2,\dots,k,
\end{align*}
where $|\cdot|$ denotes the Lebesgue measure in  $\R^N$.
\end{lemma}
\begin{proof}
  We prove the result for $m_{i,\E}^{\lambda}$, $i\geq 2$; similar
computations hold for $m_{1}^{\lambda}$.  By \eqref{e:ipotesiI} we have
$$
m_{i,\E}^{\lambda} \geq-\lambda \int_\Omega\left(\int_0^{v_i(x)}f_{i,\E}(s)ds\right) dx
\geq -\lambda\frac{\alpha}{k}|\Omega|.
$$
We are left to show that
  for any $\delta>0$ there exists $\phi\in H^1_0(\Omega)$ such that
$$
\lim_{\lambda\to\infty} \lambda^{-1}J_{i,\E}^{\lambda}(\phi)\leq
-\frac{\alpha}{k}|\Omega|+\delta.
$$
Let $\phi\in H^1_0(\Omega)$ such that
$\phi\geq 0$ a.e. in $\Omega$ and
$\phi=\beta/\sqrt{k}$ on a set $\Omega'\subset\Omega$ satisfying
$|\Omega'|>|\Omega|-\delta$. Then
 \begin{align*}
   \lambda^{-1}J_{i,\E}^{\lambda}(\eps_i\phi)&
   \leq\lambda^{-1}\frac{\eps_i^2}2
   \int_\Omega|\nabla\phi|^2-\frac{\alpha}{k}|\Omega'| \leq
   C\lambda^{-1}+\frac{\alpha}{k}(-|\Omega|+\delta)
 \end{align*}
for some $C>0$, and the result follows for $\lambda$ large.
\end{proof}
According to the lemma above, we define $\lambda_0$ as the smallest
positive number which is greater than $\lambda_1(\Omega)$ and
for which the following inequality holds
\begin{equation}\label{eq:lambda}
\lambda^{-1}m_{1}^{\lambda}<-\alpha|\Omega|\bigg(1-\frac{1}{2k}\bigg),\qquad \forall\lambda>\lambda_0.
\end{equation}
\begin{lemma}\label{l:upper1}
 Let  $H\in C^1(\R^k)$ satisfy
  (H1--2).
 Let $(u_1,\dots,u_k)$ be a weak $[H^1_0(\Omega)]^k$-solution of the
  system (\ref{eq:sistema_con_H}). Then
$$
0\leq u_i(x)\leq \beta_i,
\quad\text{for a.e. $x\in\Omega$ and all $i=1,\dots,k$}.
$$
\end{lemma}
\begin{proof}
  Testing (\ref{eq:sistema}) with $-u_i^-$ and using (F1) and (H2), we obtain
  that $u_i\geq0$ a.e. in $\Omega$ for all $i=1,2,\dots,k$. On the
  other hand, by testing (\ref{eq:sistema}) with $(u_i-\beta_i)^+$ and
  using (\ref{e:ipotesiI}) and (H2), we deduce the required inequality.
\end{proof}
\begin{lemma}\label{l:upper}
 Let $\Omega$ be a bounded Lipschitz domain
  satisfying (D1--2).
If $u\in H^1_0(\Omega)$ weakly solves
$$
-\Delta u\leq  \lambda g(u), \quad\text{in }\Omega,
$$
then, for  a.e. $x\in\Omega$,
$$
u(x_1,\dots,x_N)\leq \gamma\bigg(x_1^2-m^2\sum_{i=2}^N x_i^2\bigg).
$$
with $\gamma=\frac{\lambda\,\max_{[0,\beta]}
  g}{2(m^2(N-1)-1)}$.
\end{lemma}
\begin{proof}
  Let us denote the right hand side of the inequality by $\tilde
  u(x_1,\dots,x_N)$.  By simple computations, noticing that $\tilde u$
  is nonnegative on the boundary of $T$ and $\Omega\subset T$, it is easy to
  verify that
$$
\begin{cases}
-\Delta (\tilde u-u) \geq \lambda\,\max_{[0,\beta]} g-\lambda g(u)
\geq 0,&\text{in }\Omega,\\
\tilde u-u \geq 0,&\text{on }\partial\Omega.
\end{cases}
$$
Testing the above inequality with $-(\tilde u-u)^-$ we deduce that
$u(x)\leq \tilde u(x)$ for a.e. $x\in\Omega$.
\end{proof}
\section{Proof of the main results.}
\subsection{Proof of Theorem \ref{th:nonexistence}}\label{pf:nonexhistence}
For $U\in\U$, let ${\mathcal E}(U)$ be defined as in
\eqref{eq:freeenergy} with $F_i(s)=\int_0^s f(t)\,dt$ for all $i$ and
let
$$
\mu=\inf_{u\in H^1_0(\Omega)} \int_\Omega\bigg(\frac12|\nabla
u(x)|^2-F(u)\bigg)dx,
$$
By taking $k$-tuples of the form $(u,0,\dots,0)$, we realize that
$$
\inf_{U\in\U\cap H^1_0(\Omega)}{\mathcal E}(U)\leq \mu.
$$
Assume there exists a minimizing $k$-tuple $V=(v_1,\dots,v_k)\in\U\cap
H^1_0(\Omega)$ and define $\tilde V=(\tilde v_i,\dots,\tilde v_k)$
where $\tilde v_1=\sum_i v_i$ and $\tilde v_i=0$ for all
$i>1$. Then
$$
\mu\geq {\mathcal E}(V)={\mathcal E}(\tilde
V)=\int_\Omega\bigg(\frac12|\nabla \tilde v_1(x)|^2-F(\tilde
v_1(x))\bigg)dx\geq \mu
$$
implying in particular that $\tilde v_1$ is a weak solution of
$-\Delta u= f(u)$ (see e.g. \cite{am}). By the Strong Maximum Principle,
we deduce that either $\tilde v_1\equiv 0$ (and then $v_i\equiv 0$ for all
$i=1,\dots,k$)
or $\tilde v_1(x)> 0$ for a.e. $x\in\Omega$ and then
 $k-1$ components of $V$ must be null.

\subsection{Proof of Theorem \ref{th:main-sistema}}
Let $\lambda_0$ as in \eqref{eq:lambda} and, for every fixed
$\lambda>\lambda_0$, let us consider the minimization problem
$$
\Lambda=\inf_{U\in [H_0^1(\Omega)]^k}I^{\lambda,\vk}_{\E}(U),
$$
where $I^{\lambda,\vk}_{\E}:[H_0^1(\Omega)]^k\to(-\infty,+\infty]$ is defined
in (\ref{eq:Ikappa}).

\smallskip\noindent
{\bf Step 1.} $ \Lambda$ is achieved.
We first observe that, by (H1) and Lemma \ref{l:limiti},
$$
I^{\lambda,\vk}_{\E}(U)\geq m_{1}^{\lambda}+\sum_{i=2}^km_{i,\E}^{\lambda}\geq -\lambda
\alpha|\Omega|\bigg(1+\frac{k-1}{k}\bigg)
$$
for all $U\in [H_0^1(\Omega)]^k$, hence $\Lambda>-\infty$.  Let
$\big\{V_n=(v_1^n,\dots,v_k^n)\big\}_{n\in\N}$ be a minimizing
sequence, i.e.  $\lim_{n\to+\infty}I^{\lambda,\vk}_{\E}(V_n)=\Lambda$.
Notice that we can choose $V_n$ such that $v_i^n\geq 0$ a.e. in
$\Omega$ for all $i=1,\dots,k$; otherwise we take
$((v_1^n)^+,\dots,(v_k^n)^+)$ with $(v_i^n)^+:=\max\{v_i^n,0\}$, which
is another minimizing sequence. Indeed, in view of (\ref{e:ipotesiI})
and (H2), the function $t\mapsto
H(s_1,\dots,s_{i-1},t,s_{i+1},\dots,s_k)$ has a global minimum in
$t=0$ thus yielding $H(v_1^+,\dots,v_k^+)\leq
H(v_1,\dots,v_k)$. Besides, since the function $t\mapsto
H(s_1,\dots,s_{i-1},t,s_{i+1},\dots,s_k)$ is non decreasing in
$(0,+\infty)$ from (F2) and (H2), letting $U_n=(u_1^n,\dots,u_k^n)$ with
$u_i^n=\min\{v_i^n,\beta_i\}$, we have that
$I^{\lambda,\vk}_{\E}(U_n)\leq I^{\lambda,\vk}_{\E}(V_n)$.  Then also
$\big\{U_n\big\}_{n\in\N}$ is a minimizing sequence.

Since $\big\{U_n\big\}_{n\in\N}$ is a minimizing sequence and it is
uniformly bounded, it is easy to realize that
$\big\{U_n\big\}_{n\in\N}$ is bounded in $\big[H^1_0(\Omega)\big]^k$;
hence there exists a subsequence, still denoted as
$\big\{U_n\big\}_{n\in\N}$, which converges to some
$U=(u_1,\dots,u_k)\in \big[H^1_0(\Omega)\big]^k$ weakly in
$\big[H^1_0(\Omega)\big]^k$, strongly in $\big[L^2(\Omega)\big]^k$ and
a.e. in $\Omega$.  A.e.  convergence implies that $0\leq u_i\leq
\beta_i$ a.e. in $\Omega$.  From the Dominated Convergence Theorem, it
follows that
\begin{align*}
  &\lim_{n\to+\infty}\int_{\Omega}
  F_{i,\E}(u_i^n(x))\,dx=\int_{\Omega} F_{i,\E}(u_i(x))\,dx,
  \quad\text{for every $i=1,\dots,k$},\\
  &\lim_{n\to+\infty}\int_{\Omega} H(u_1^n(x), \dots, u_k^n(x))\,dx
  =\int_{\Omega} H(u_1(x), \dots, u_k(x))\,dx,
\end{align*}
 which, together with  weak lower
semi-continuity, yields
\[
\Lambda\leq I^{\lambda,\vk}_{\E}(U)\leq
\liminf_{n\to+\infty}I^{\lambda,\vk}_{\E}(U_n)=
\lim_{n\to+\infty}I^{\lambda,\vk}_{\E}(U_n)=\Lambda,
\]
thus proving that $U$ attains $\Lambda$.

\smallskip\noindent {\bf Step 2.} If $U=(u_1,\dots,u_k)$ is a
minimizer attaining $\Lambda$, then
\begin{equation}\label{eq:gamma}
  u_1\not\equiv 0,\qquad\mbox{and}\qquad J_1^\lambda(u_1)
  <-\frac{\alpha\lambda|\Omega|}{2k}<0.
\end{equation}
Indeed, letting $u\in H^1_0(\Omega)$ such that $ m_{1}^{\lambda}=J_1^\lambda(u)$,
 we have
$$\Lambda\leq I^{\lambda,\vk}_{\E}(u,0,\dots,0)= J_1^\lambda (u)=m_{1}^{\lambda}.$$
Besides, appealing to Lemma \ref{l:limiti} we learn that
$$
\sum_{i=2}^kJ_{i,\E}^{\lambda}(u_i)\geq
\sum_{i=2}^km_{i,\E}^{\lambda}\geq -\alpha\lambda|\Omega|\frac{k-1}k.
$$
Since $\Lambda=\Ik(U)\geq
J_{1}^{\lambda}(u_1)+\sum_{i=2}^kJ_{i,\E}^{\lambda}(u_i) $ by (H1),
choosing $\lambda$ as in  (\ref{eq:lambda}) we finally have
\begin{equation*}
J_1^{\lambda}(u_1)\leq m_{1}^{\lambda}-\sum_{i=2}^kJ_{i,\E}^{\lambda}(u_i)
<-\alpha\lambda|\Omega|\bigg(1-\frac1{2k}-\frac{k-1}k\bigg)=
-\frac{\alpha\lambda|\Omega|}{2k}<0,
\end{equation*}
thus proving the estimate in (\ref{eq:gamma}), which in particular ensures
$u_1\not\equiv 0$.

\smallskip\noindent {\bf Step 3.}  If $U=(u_1,\dots,u_k)$ is a
minimizer attaining $\Lambda$, then $u_i\not\equiv 0$ for
all $i=1,\dots,k$. We already know by step 2 that $u_1\not\equiv 0$.
Moreover, by standard Critical
Point Theory, $U=(u_1,\dots,u_k)$ is a weak  solution to
(\ref{eq:sistema_con_H}), and hence, by Lemma \ref{l:upper1},
 $0\leq u_i(x)\leq\beta_i$ for all $i=1,\dots,k$
and a.e. $x\in\Omega$.
Assume by contradiction that
\begin{equation}\label{eq:imanca}
u_i\equiv 0\qquad\mbox{ for some }i>1.
\end{equation}
Let us fix $x_0\in\Omega$ and $r,R>0$ such that
$B(x_0,r)\subset\Omega\subset B(0,R)$ and define
$$w_i(x)=\frac{\eps_i}{\sqrt{k}}u_1({\eps_i}^{-1}(x-x_0)).
$$
Notice that $w_i\in H^1_0(A_i)$ where
$$
A_i=\bigg\{x\in\R^N: \frac{x-x_0}{\eps_i}\in\Omega\bigg\}.
$$
Moreover $|A_i|=\eps_i^N|\Omega|$ and
$A_i\subset B(x_0,R\eps_i)$,
which implies $A_i\subset\Omega$ if $ \eps_i<\frac{r}{R}$.
In particular $w_i\in H^1_0(\Omega)$ if $\eps_i<\frac{r}{R}$.
From Lemma \ref{l:upper1},
\begin{equation}\label{eq:uppwi}
0\leq w_i(x)\leq \frac{\beta\,\eps_i}{\sqrt{k}}
\quad\text{for a.e. }x\in A_i.
\end{equation}
Since
$$
\int_0^{w_i(x)}f_{i,\E}(s)ds=\frac{1}{\sqrt{k}\eps_i}\int_0^{w_i(x)}
g(\sqrt{k}\,\eps_i^{-1}s)ds=
\frac{1}{k}\int_0^{u_1(\eps_i^{-1}(x-x_0))}g(t)dt,
$$
from (\ref{eq:gamma}) it follows
\begin{align}\label{eq:Ji}
  J_{i,\E}^\lambda(w_i)&=\int_{A_i}\bigg(\frac12|\nabla w_i(x)|^2-
  \lambda\int_0^{w_i(x)}f_{i,\E}(s)ds\bigg)dx\\\nonumber
  &=\int_{A_i}\bigg(\frac1{2k}\bigg|\nabla
  u_1\bigg(\frac{x-x_0}{\eps_i}\bigg)\bigg|^2-
  \frac\lambda k\int_0^{u_1(\eps_i^{-1}(x-x_0))}
  g(s)ds\bigg)dx\\\nonumber
  &=\frac{\eps_i^N}{k}\int_\Omega\bigg(\frac12|\nabla
  u_1(x)|^2-\lambda\int_0^{u_1(x)}g(s)ds\bigg)dx\\\nonumber
  &=\frac{\eps_i^N}{k}J_1^\lambda(u_1)
<-\eps_i^N\frac{\alpha\,\lambda\,|\Omega|}{2k^2}.
\end{align}
We now claim that, letting
$$
W=(u_1,\dots,u_{i-1},w_i,u_{i+1},\dots,u_k),
$$
there holds $\Ik(W)<\Ik(U)$ provided $\eps_i$ is small enough.
Indeed
\begin{align}\label{eq:W-U}
  \Ik(W)-\Ik(U)=J_{i,\E}^\lambda(w_i)+ \k\int_{A_i}\big(H(W(x))
  -H(U(x))\big)\,dx.
\end{align}
From (H2) it follows that the function $t\mapsto
H(s_1,\dots,s_{i-1},t,s_{i+1},\dots,s_k)$ is non decreasing in $(0,+\infty)$,
hence, in view of (\ref{eq:uppwi}),
\begin{align}\label{eq:stimHWU}
0&\leq H(W(x))
  -H(U(x))\\
\notag&=H(u_1(x),\dots,u_{i-1}(x),w_i(x),u_{i+1}(x),
\dots,u_k(x))\\
\notag&\hskip4cm-H(u_1(x),\dots,u_{i-1}(x),0,u_{i+1}(x),\dots,u_k(x))\\
\notag&\leq H\Big(u_1(x),\dots,u_{i-1}(x),\frac{\beta\eps_i}{\sqrt{k}},u_{i+1}(x),
\dots,u_k(x)\Big)\\
\notag&\hskip4cm-H(u_1(x),\dots,u_{i-1}(x),0,u_{i+1}(x),\dots,u_k(x)).
\end{align}
Since the restriction of $H$ to the cube $[0,\beta]^k$ is uniformly
continuous, it admits a modulus of continuity, i.e. there exists a function
$\omega:[0,+\infty)\to[0,+\infty]$ such that $\lim_{t\to0^+}\omega(t)=0$
and $|H(X)-H(Y)|\leq \omega(|X-Y|)$ for all $X,Y\in [0,\beta]^k$.
Hence, from (\ref{eq:stimHWU}) we derive
\begin{align}\label{eq:stHWU2}
0\leq H(W(x))
  -H(U(x))\leq \omega\bigg(\frac{\beta\eps_i}{\sqrt{k}} \bigg).
\end{align}
Combining (\ref{eq:Ji}), (\ref{eq:W-U}), and (\ref{eq:stHWU2}), we obtain that
$$
  \Ik(W)-\Ik(U)<-\eps_i^N\frac{\alpha\,\lambda\,|\Omega|}{2k^2}
+\varkappa{\eps_i}^N|\Omega|\omega\bigg(\frac{\beta\eps_i}{\sqrt{k}} \bigg)
$$
which is strictly negative if $\eps_i$ is small enough to
satisfy
\begin{equation}\label{eq:epsi}
\omega\bigg(\frac{\beta\eps_i}{\sqrt{k}} \bigg)<\frac{\alpha\,\lambda}
{2k^2\varkappa}.
\end{equation}
This concludes the proof.

\subsection{Proof of Theorem \ref{th:main-partition}}
Let $k\geq 2$ and $\lambda\geq\lambda_0$ be fixed as in \eqref{eq:lambda}.\\
Assume by contradiction that the minimum of ${\mathcal
  E}^{\lambda}_{\E}$ on $\mathcal{U}\cap H^1_0(\Omega)$ is achieved by a $k$--tuple
$U=(u_1,\dots,u_k)$ with only one nontrivial component. Reasoning as in
\eqref{eq:gamma} it is easy to prove that $u_1\not\equiv 0$, hence we
can assume $u_j\equiv 0$ for all $j>1$; notice that $u_1$ is in
particular a global minimizer of $J^\lambda_1$.  The strategy leading
a contradiction consists in modifying $u_1$ near the origin in order to
create a new $k$-tuple $V\in\mathcal{U}$ with a second non--vanishing
component and
 $${\mathcal E}^{\lambda}_{\E}(V)<{\mathcal E}^{\lambda}_{\E}(U)$$
 for $\E$ small.
To this aim,
let $\phi\in C^{2}(\R)$ be a cut--off function such that $0\leq\phi\leq 1$ and
$$\phi(s)=\begin{cases} 0&  s \leq 1,\\
1 & s\geq 2.
\end{cases}
$$
Given $\delta\in (0,\delta_0)$ we set
$$
\Omega_\delta=\{x\in \Omega: x_1<\delta\},\qquad \Omega_\delta'=\{x\in
\Omega: \delta < x_1<2\delta\}.
$$
Let us  define
$$
u_{1,\delta}(x)=\phi(\delta^{-1}x_1)u_1(x),
$$
which vanishes on $\Omega_\delta$ and belongs to $H^1_0(\Omega)$
since, for $x\in \Omega_{2\delta}$,
$$
\nabla u_{1,\delta}(x)=
\delta^{-1}\phi'(\delta^{-1}x_1)u_1(x)(1,0,\dots,0)
+\phi(\delta^{-1}x_1)\nabla u_1(x).
$$
The growth of energy occurring when substituting $u_1$ in the
minimizing $k$-tuple with $u_{1,\delta}$ can be estimated as follows.
Observing first that  $F_1(u_{1,\delta})>0$ by Lemma \ref{l:upper1},
and $F_1(s)\leq \Gamma s^2$ for some $\Gamma>0$ by (F1--2), we have
\begin{align*}
&J^\lambda_1(u_{1,\delta})-J^\lambda_1(u_1)\\
&\leq\int_{\Omega_{2\delta}}\frac12\bigg(|\nabla u_{1,\delta}(x)|^2-|\nabla u_1(x)|^2\bigg)\,dx+
\lambda\int_{\Omega_{2\delta}}F_1(u_1(x))\,dx\\
&\leq \int_{\Omega_{\delta}'}\bigg(\frac12\delta^{-2}\big(\phi'(\delta^{-1}x_1)\big)^2u_1^2(x)
+\delta^{-1}\phi(\delta^{-1}x_1)\phi'(\delta^{-1}x_1)u_1(x)\frac{\partial }{\partial x_1}u_1(x)\bigg)\,dx\\
&\qquad+\Gamma\lambda\int_{\Omega_{2\delta}}u_1^2(x)\,dx.
\end{align*}
An integration by parts provides
\begin{align*}
  2&\int_{\Omega_{\delta}'}\phi(\delta^{-1}x_1)
  \phi'(\delta^{-1}x_1)u_1(x)\frac{\partial}{\partial x_1}u_1(x)\,dx=
  \int_{\Omega_{\delta}'}\phi(\delta^{-1}x_1)\phi'(\delta^{-1}x_1)
  \frac{\partial}{\partial x_1}u_1^2(x)\,dx\\
  &=
  \int_{\partial\Omega_{\delta}'}\phi(\delta^{-1}x_1)\phi'(\delta^{-1}x_1)
  u_1^2(x)\,d\sigma\\
  &\qquad-\frac{1}{\delta}
  \int_{\Omega_{\delta}'}\bigg((\phi'(\delta^{-1}x_1))^2+
  \phi(\delta^{-1}x_1)\phi''(\delta^{-1}x_1)\bigg)u_1^2(x)\,dx\\
  &\leq \|\phi'\|_{L^\infty(\R)}\int_{\partial\Omega_{\delta}'}
  u_1^2(x)\,d\sigma+\frac{1}{\delta}(\|\phi'\|^2_{L^\infty(\R)}+
  \|\phi''\|_{L^\infty(\R)})\int_{\Omega_{\delta}'}u_1^2(x)\,dx\\
  &\leq M \int_{\partial\Omega_{\delta}'}u_1^2(x)\,d\sigma
  +\frac{M}{\delta}\int_{\Omega_{\delta}'}u_1^2(x)\,dx,
\end{align*}
for some $M=M(\phi)>0$.  Appealing to Lemma \ref{l:upper} we have that
$u_1(x)\leq \gamma x_1^2\leq 4\gamma\delta^2 $ for all $x\in
\Omega_{2\delta}$. Hence we have the following estimate:
\begin{align*}
  J^\lambda_1(u_{1,\delta})&-J^\lambda_1(u_1)\\
  &\leq |\Omega_{2\delta}|\left(\frac12\delta^{-2}M(4\gamma\delta^2)^2
    +\Gamma\lambda(4\gamma \delta^2)^2+\frac12\delta^{-2}M(4\gamma
    \delta^2)^2\right)\\
&\qquad  +\frac12\delta^{-1}M|\partial \Omega_{\delta}'|_{N-1}(4\gamma\delta^2)^2\\
  &\leq C\delta^{N+2},
\end{align*}
for some $C>0$.  Now fix any $i>1$, set
\begin{equation}\label{eq:lev}
v_i(x)=\frac{\eps_i}{\sqrt{k}}u_1({\eps_i}^{-1}x),
\end{equation}
and define
$V=(v_1,\dots,v_k)$ where
\begin{equation}\label{eq:leu}
v_1(x)=u_{1,\eps_i}(x)=\phi(\eps_i^{-1}x_1)u_1(x),
\end{equation}
and $v_j\equiv 0$ if $j\neq 1,i$.
Notice that $v_1\cdot v_i=0$ by construction, so that $V\in\mathcal{U}$ and hence $H(v_1,\dots,v_k)=0$ by (H3).
Besides, by the above computations with $\delta=\eps_i$ and arguing as in
\eqref{eq:Ji} to estimate $J^\lambda_{i,\E}(v_i)$, we have
\begin{align}\label{eq:scendo}
{\mathcal E}^{\lambda}_{\E}(V)-{\mathcal E}^{\lambda}_{\E}(U)&=
J^\lambda_1(v_1)-J^\lambda_1(u_1)+J^\lambda_{i,\E}(v_i)\\
\nonumber &\leq
\frac{\eps_i^N}{k}J^\lambda_1(u_1)+C\eps_i^{N+2}\\\nonumber &\leq
\eps_i^N\bigg(-\frac{\alpha\,\lambda\,|\Omega|}{2k^2}+C\eps_i^{2}\bigg),
\end{align}
which is strictly negative for $\eps_i$ small and provides a contradiction.
\subsection{Proof of Theorem \ref{th:system2}}$\;$\\
\textbf{Coexistence.}  Let $k=2$ and $\vk>0$ be fixed.  Arguing as in
the proof of Theorem \ref{th:main-sistema}, we immediately obtain
that the global minimum of $\Ik$ is achieved for all $\vk>0$. Assume by
contradiction that the global minimum of $\Ik$ on $[H^1_0(\Omega)]^2$
is achieved by a pair of the form $U^\vk=(u_1^\vk,0)$ (recall that by
\eqref{eq:gamma} the first component of any minimizer must be
nontrivial). Arguing exactly as in the proof of Theorem
\ref{th:main-partition}, we define a new pair
$V^\vk=(v_1^\vk,v_2^\vk)$ by setting
$$
v_1^\vk(x)=\phi(\eps_2^{-1}x_1)u_1^\vk(x),\qquad v_2^\vk(x)=
\frac{\eps_2}{\sqrt{2}}u_1^\vk({\eps_2}^{-1}x),
$$
as in \eqref{eq:leu} and \eqref{eq:lev} respectively.  Since $v_1^\vk$
and $v_2^\vk$ have disjoint supports, by (H3) it holds $H(v_1^\vk,v_2^\vk)=0$
and no interaction term appears in the evaluation of
$\Ik(V^\vk)$. Hence
$$
\Ik(V^\vk)-\Ik(U^\vk)=J^\lambda_1(v_1^\vk)-J^\lambda_1(u_1^\vk)+J^\lambda_{2,\E}(v_2^\vk),
$$
and estimating as in \eqref{eq:scendo}, we find that
$\Ik(V^\vk)-\Ik(U^\vk)$ is strictly negative for $\eps_2$ sufficiently
small, a contradiction.  Hence, if $\eps_2$ is small enough, for any
positive value of $\vk$ the competing system \eqref{eq:sistema_con_H}
has a solution $U^\vk=(u_1^\vk,u_\vk^2)$ with both nonzero components
which minimizes the energy $\Ik$.

\smallskip\noindent \textbf{Asymptotic analysis.}  Let $\lambda$ and
$\E$ be fixed and consider $U^{\vk}=(u_1^{\vk},u_2^{\vk})$ such that
$$\Ik(U^{\vk})=\Lambda^{\vk}=\inf_{U\in [H_0^1(\Omega)]^2}I^{\lambda,\vk}_{\E}(U).$$
The convergence of $U^{\vk}$ to a minimizer of ${\mathcal E}^\lambda_{\E}$ on
$\mathcal{U}\cap H^1_0(\Omega)$ can be proven as in \cite[Theorem
2.3]{cf-neumann}, with minor changes. For the reader's convenience, we
report some details.  Notice first that evaluating $\Ik(U)$ for all
$U\in\mathcal{U}$ annihilates the interaction term in light of (H3),
so that
\begin{equation}\label{eq:8}
\Lambda^{\vk}\leq \min\{ {\mathcal E}^\lambda_{\E}(U),\;U\in\mathcal{U}\}=:c,
\end{equation}
hence we have
\[
\|U^{\vk}\|_{[H^1_0(\Omega)]^2}^2\leq 2
\Lambda^{\vk}+2|\Omega|\lambda \sum_i \int_0^{\beta_i} f_{i,\E}(s)ds \leq
2c+3|\Omega|\lambda\alpha.
\]
Then $u_i^{\vk}$ is bounded in $H^1_0(\Omega)$ uniformly with respect
to $\vk$, and there exists $u_i\geq 0$ such that, up to subsequences,
$u_i^{\vk}\rightharpoonup u_i$ weakly in $H^1_0(\Omega)$ and
$u_i^{\vk}(x)\to u_i$ for almost every $x$ as $\vk\to+\infty$.  Let us
now multiply the equation of $u_i^{\vk}$ times $u_i^{\vk}$ on account
of the boundary conditions: then $
\vk\int_{\Omega}u_i^\vk\frac{\partial H}{\partial
  s_i}(u^{\vk}_1,u^{\vk}_2)$ is bounded uniformly in $\vk$,  giving
\[
\int_{\Omega}u_i^\vk(x)\frac{\partial H}{\partial
  s_i}(u^{\vk}_1(x),u^{\vk}_2(x))\,dx\to 0,\qquad \text{as
}\vk\to\infty.
\]
By (H2) and the Dominated Convergence Theorem (recall that $0\leq
u_i^{\vk}\leq \beta_i$) we infer that $$u_i(x)\frac{\partial
  H}{\partial u_i}(u_1(x),u_2(x))=0 \qquad \text{a.e. } x\in\Omega,$$
implying in light of (H3) that $u_1(x)\cdot u_2(x)=0$ and hence
$U=(u_1,u_2)\in\U$.  Now notice that for $\vk\leq \vk'$ it holds
$\Lambda^{\vk}\leq \Lambda^{\vk'}\leq c$, hence the following chain of
inequalities holds:
\begin{align*}
  c&\geq\lim_{\vk\to\infty}\Lambda^{\vk}
  =\lim_{\vk\to\infty}\Ik(U^{\vk})\\
  & = \limsup_{\vk\to\infty}\bigg[
  \sum_{i=1}^2\bigg\{\frac12\int_\Omega  |\nabla u_i^{\vk}(x)|^2 dx
  -\lambda\int_{\Omega} F_{i,\E}(u_i^{\vk}(x))\,dx\bigg\}\\
&\hskip6cm+
  \vk\int_{\Omega}H(u^{\vk}_1(x),u^{\vk}_2(x))dx\bigg]\\
  &\geq \limsup_{\vk\to\infty}
  \sum_{i=1}^2\bigg\{\frac12\int_\Omega|\nabla u_i^{\vk}(x)|^2 dx
  -\lambda\int_{\Omega} F_{i,\E}(u_i^{\vk}(x))\,dx\bigg\}\\
  &\geq \liminf_{\vk\to\infty}
  \sum_{i=1}^2\bigg\{\frac12\int_\Omega|\nabla u_i^{\vk}(x)|^2 dx
  -\lambda\int_{\Omega}F_{i,\E}(u_i^{\vk}(x))\,dx\bigg\}\\
  &\geq\sum_{i=1}^2\bigg\{\frac12\int_\Omega|\nabla u_i(x)|^2 dx
  -\lambda\int_{\Omega} F_{i,\E}(u_i(x))\,dx\bigg\}= {\mathcal E}^\lambda_{\E}(U)\geq
  c.
\end{align*}
Therefore all the above inequalities are indeed equalities. In
particular ${\mathcal E}^\lambda_{\E}(U)=c$, meaning that $U$ is a global
minimizer of ${\mathcal E}^\lambda_{\E}$.

Moreover, we learn that $\lim_{\vk\to+\infty}\|U^{\vk}\|_{[H^1_0(\Omega)]^2}=
\|U\|_{[H^1_0(\Omega)]^2}$ which implies that the weak convergence of
$U^{\vk}$ to $U$ is
actually strong in $[H^1_0(\Omega)]^2$.

Finally, to prove that both the components of $V$ are positive, we
appeal to Theorem \ref{th:main-partition} in the case $k=2$, ensuring
that for $\eps_2$ small any global minimizer of ${\mathcal E}^{\lambda}_{\E}$ on
$\mathcal{U}\cap [H^1_0(\Omega)]^2$ has two nontrivial components.


\end{document}